\documentclass[letterpaper, 10 pt, conference]{ieeeconf}  
\IEEEoverridecommandlockouts                              
\usepackage{amsmath,amsfonts,graphicx,subfigure}
\usepackage{url}
\usepackage{subfigure}

\newtheorem{lemma}{Lemma}
\newtheorem{theorem}{Theorem}
\newtheorem{corollary}{Corollary}
\newtheorem{example}{Example}

\begin{document}
\title{\LARGE \bf A Technique for Designing Stabilizing Distributed Controllers with Arbitrary Signal Structure Constraints}
\author{Anurag Rai and Sean Warnick\\Information and Decision Algorithms Laboratories\\Brigham Young University\\{\tt anurag1985@gmail.com, sean.warnick@gmail.com}}
\maketitle

\begin{abstract}
This paper presents a new approach to distributed controller design that exploits a partial-structure representation of linear time invariant systems to characterize the structure of a system.  This partial-structure representation, called the dynamical structure function, characterizes the {\em signal structure}, or open-loop causal dependencies among manifest variables, capturing a significantly richer notion of structure than the sparsity pattern of the transfer function.  The design technique sequentially constructs each link in an arbitrary controller signal structure, and the main result proves that the resulting controller is either stabilizing or no controller with the desired structure can stabilize the system.  
\end{abstract}

\section{Introduction: The Meaning of Structure}

Distributed controller design concerns the imposition of architectural constraints on a feedback controller while attempting to stabilize, and possibly optimize, the closed-loop performance of a given system, called the {\em plant}.  The problem only arises when the plant is multi-input and multi-output, and the standard notion of architectural constraints implies that certain elements of the controller transfer function matrix are forced to be zero.

Although the sparsity pattern of a transfer function is certainly one notion of a system's structure, it is typically the weakest form of system structure considered.  There are other notions of system structure, such as the interconnection pattern of subsystems or the sparsity pattern of a state space realization that are stronger structural concepts \cite{enoch:math_rel1,enoch:math_rel2}.  Here we say they are stronger structural concepts in the sense that the interconnection of subsystems or a particular state space realization determines the sparsity pattern of the associated transfer function, but not the other way around.  

In this paper we consider another notion of structure, called the signal structure of the system, that is both stronger than the sparsity pattern of the transfer function but weaker than the sparsity pattern of the system's state space realization.  If we use these two system representations as extremes, suggesting that the sparsity pattern of the state realization is the {\em complete computational structure} of the system while the sparsity pattern of the transfer function may contain little (if any) structural information, then the signal structure is squarely between the two in terms of its structural informativity.
 
The signal structure is encoded by a representation of linear time invariant systems called the dynamical structure function (DSF) \cite{TAC08}.  Since all representations of the system, whether a state realization, DSF, or transfer function, describe the system's behavior or dynamic response to inputs equally well, these representations really differ in how much structural information they convey about the system.  As a result, the DSF is a {\em partial-structure} representation of the system.

Although these ideas will be made precise in the sequel, intuitively the system's DSF describes the {\em open-loop} causal dependencies among manifest variables (inputs and outputs), whereas the transfer function describes the {\em closed-loop} dependencies from inputs to outputs.  Thus, while a DSF may be intricately structured, its corresponding transfer function may be fully connected, essentially exhibiting no particular structure (see Figure 1).  This is why many interesting distributed control problems are not described well by imposing sparsity constraints on the controller's transfer function.

\begin{figure}[ht]
\centering
\subfigure[]{
\includegraphics[trim = 0 .4in 0 .2in, scale=.33]{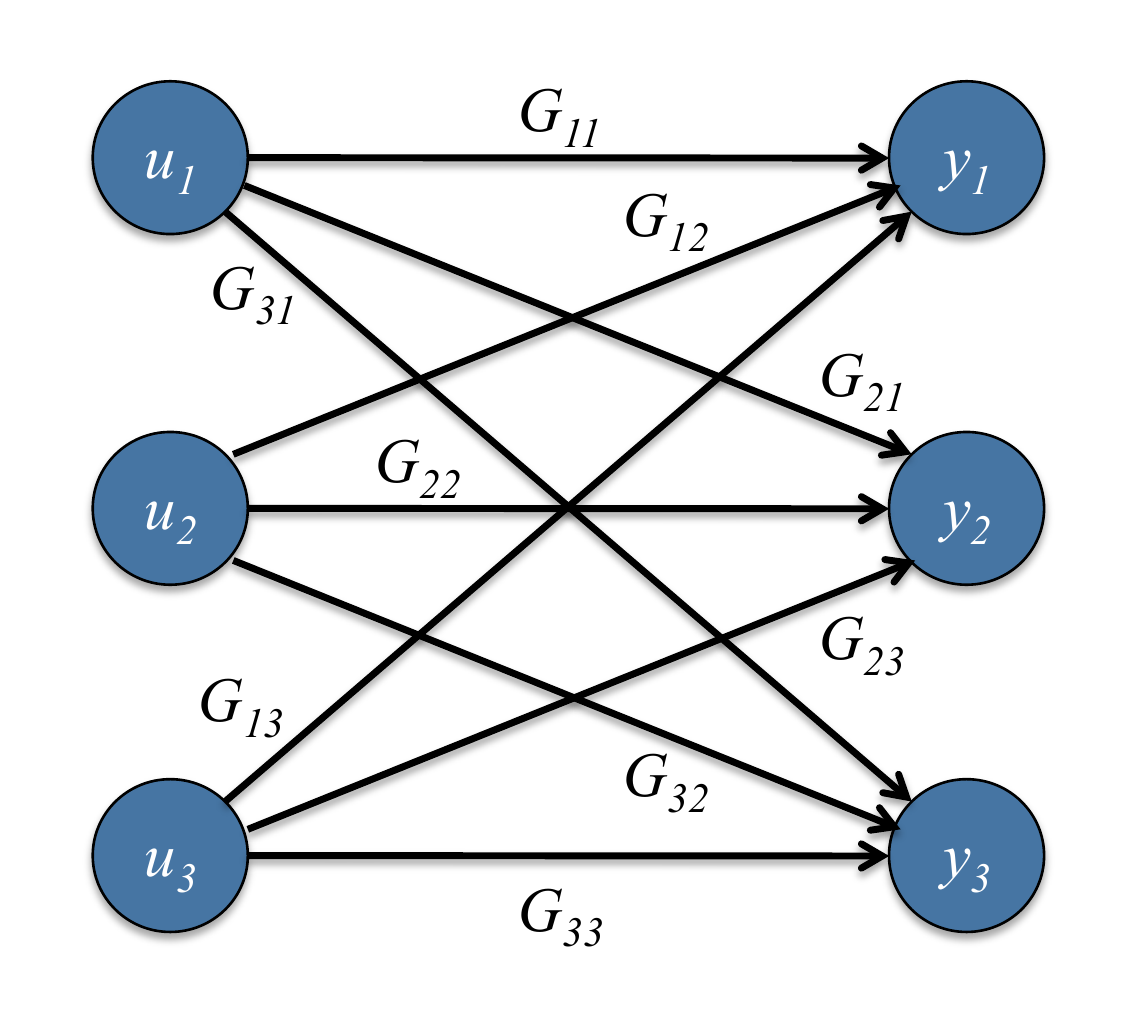}
\label{fig:ctrl_tf}
}
\subfigure[]{
\includegraphics[trim = 0 .4in 0 .2in, scale=.43]{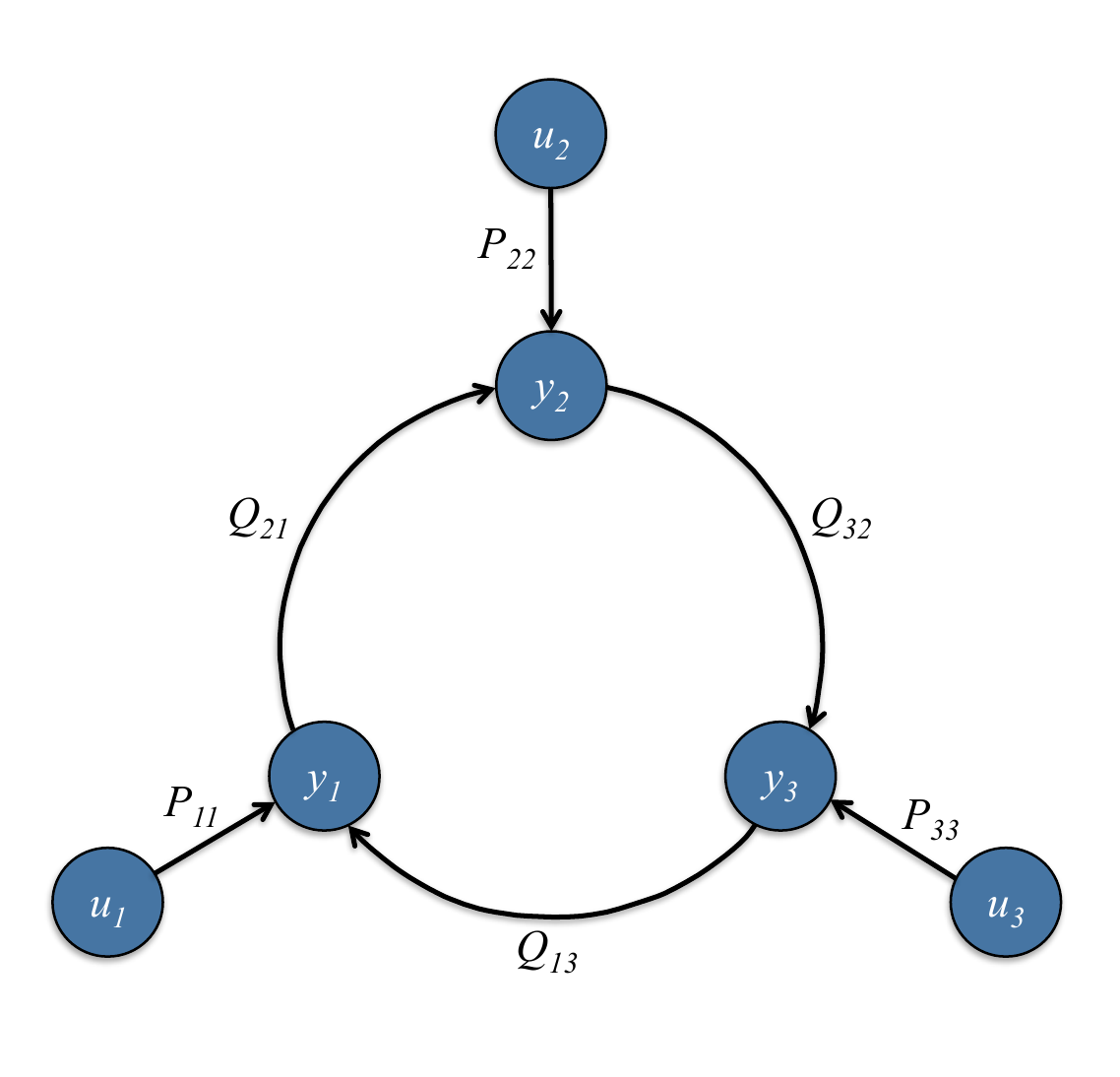}
\label{fig:ctrl_dsf}
}
\caption{Two distinct notions of structure for the same system.  The top figure indicates that the transfer function, evidently a $3\times 3$ matrix $G(s)$, is full and unstructured, while the bottom figure indicates that the signal structure, represented by the dynamical structure function with two $3\times 3$ matrices $Q(s)$ and $P(s)$ where $G(s)=(I-Q(s))^{-1}P(s)$,  is sparse and definitively structured.  Note that the bottom figure may represent communication links, and since there is a pathway from every input to every output, the associated transfer function may be full, as in the top figure.}
\label{fig:ctrl_required}
\end{figure}

This paper describes a technique for designing stabilizing controllers with a particular signal structure for a given plant, or demonstrating that no such controller exists.  The next section discusses related work, while the following section details mathematical preliminaries regarding dynamical structure functions as a partial structure representation of linear time invariant systems.  We then present the design procedure and the main result, which proves that the design procedure delivers a stabilizing controller with the desired structure if possible.  Examples and conclusions follow.

\section{Related Work}

One of the first results on the existence of a decentralized controller was given in \cite{wang_davison}. It developed the idea of {\em fixed modes} and showed that a decentralized controller exists if and only if the system had no unstable fixed modes. More precisely, it showed that a system $(A,B,C)$ is stabilizable with a diagonal or block diagonal controller $K$  if and only if $A-BKC$ does not have any unstable eigenvalues that cannot be moved by changing the nonzero entries of $K$. This result was extended in \cite{siljak:decentralized_control} by showing that this is in fact true for any distributed controller  $K$, not just for diagonal and block diagonal. The authors also present methods to synthesize the decentralized stabilizing controller.

In \cite{lall:qi} the authors show that if the structure of the transfer function matrices of the plant and the controller meets the {\em quadratic invariance} condition then the problem of synthesizing the optimal controller is convex. In \cite{lall:qi2} the authors show that the quadratic invariance condition is necessary and sufficient for the problem of synthesizing the optimal controller to be convex. This method requires a decentralized stabilizing controller to initialize the convex optimization problem, so to complete the process, an algorithm to obtain such a controller is provided in \cite{nuno:qi}.

A different type of distributed controller design has been proposed in \cite{nicola:reliazable_ctrl}. The approach taken in this paper enforces the controller to have the same network structure as the plant. The structure in this paper is defined as the constraint on the interconnection of sub-systems, or the subsystem structure. Hence, the plant and the controller can share the same communication network reducing the implementation cost. An algorithm to synthesize a sub-optimal controller with such structure  is also provided in this paper.

In this paper we introduce a similar, but a more general controller design problem. Instead of the controller having to have the same structure as the plant, we allow it to have any structure. Also, the structure is defined as a constraint on the signal structure. In Figure \ref{fig:ctrl_required} we show an example of a plant and a corresponding controller structure that we might want to have. When a controller has such a structure, we can see that all the controller units affect each other directly or indirectly, hence, the controller transfer function matrix is completely full. As a result, using the usual approach of placing binary constraint on the controller transfer function will produce a centralized controller as shown in \ref{fig:ctrl_obtained}. Also, most of these setups do not meet the quadratic invariance criterion. In this paper, we will show that these controllers can be obtained by placing binary constraints on the dynamical structure function of the controller.

\begin{figure}[ht]
\centering
\includegraphics[scale=.5]{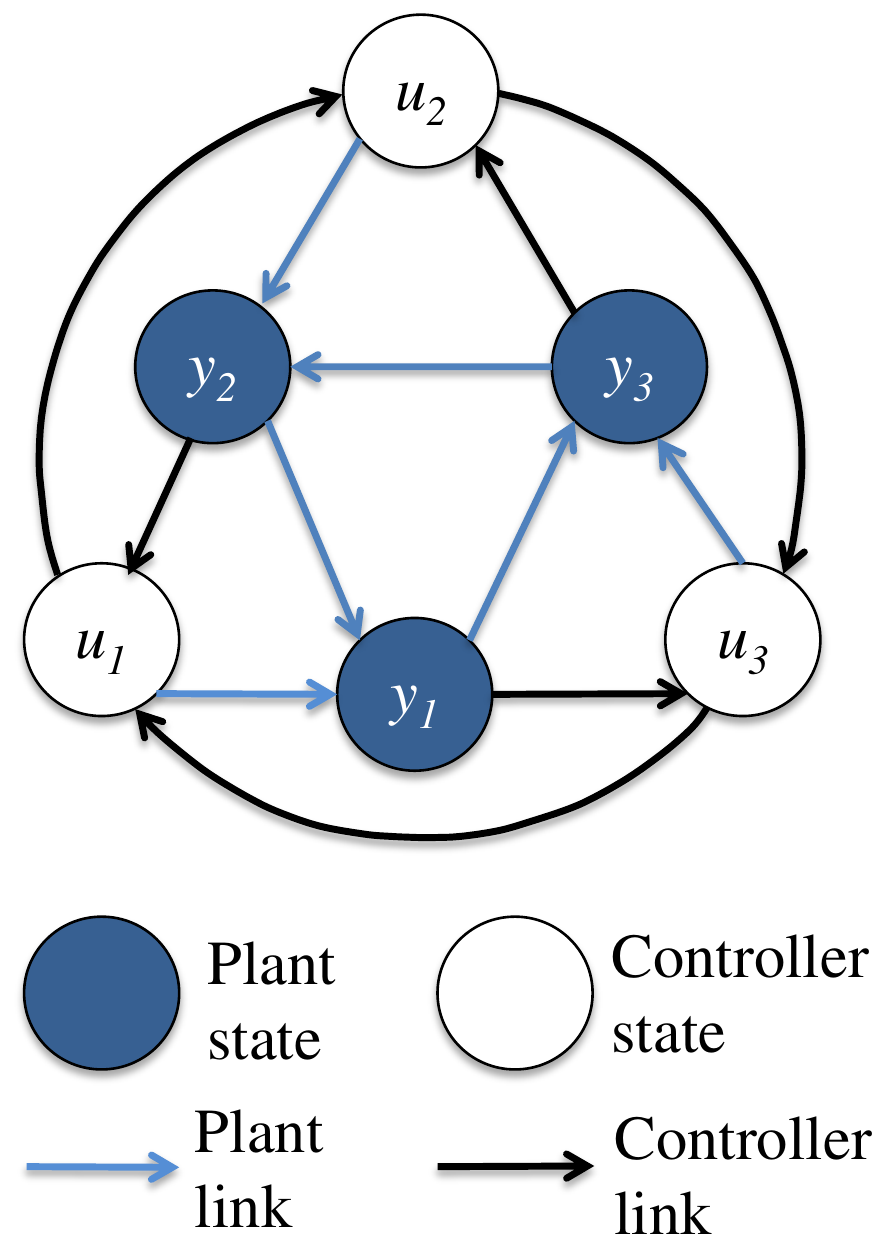}
\caption{Plant with the signal structure as in Figure \ref{fig:ctrl_dsf} interconnected with controller with a particular desired distributed structure.}
\label{fig:ctrl_required}
\end{figure}

\begin{figure}[ht]
\centering
\includegraphics[scale = .5]{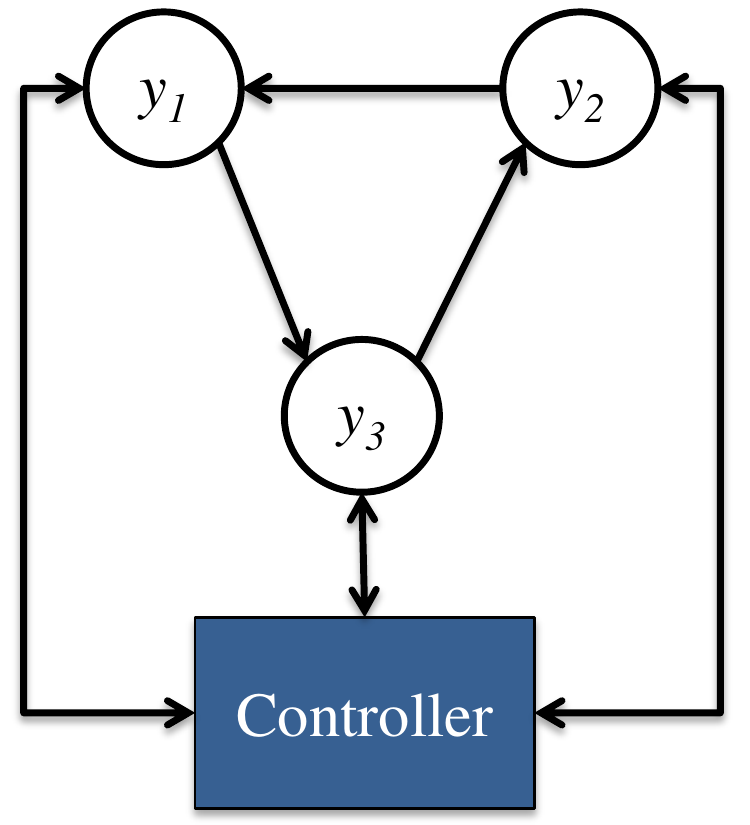}
\caption{Since the desired signal structure for the controller in Figure \ref{fig:ctrl_required} yields a full transfer function, other design methods yield a centralized controller.}
\label{fig:ctrl_obtained}
\end{figure}

In \cite{sequential1}, \cite{sequential2}, etc., sequential design methods have been used to construct decentralized controllers. Although these methods do not produce the optimal controller, they provide an efficient method to synthesize a nominal stabilizing controller with a desired decentralized sparsity pattern in its transfer function. We will use a similar strategy to design a stabilizing controller with constraints on the signal structure in Section \ref{sec:main}. In the event that this process cannot produce a stabilizing controller, we will show that there is no controller of the given signal structure that stabilizes the plant.

\section{dynamical structure functions}
dynamical structure functions is a representation for linear time invariant systems developed in \cite{sean:dsf}. It gives a partial representation of the structure of the system, namely how the inputs affect the manifest states and how the manifest states affect each other. We also call it this representation the signal structure of the system. A brief derivation is provided below.

Let us consider a state-space LTI system
\begin{align}\label{eqn:trans_sys}\begin{bmatrix}\dot{y} \\ \dot{x} \end{bmatrix} &= \begin{bmatrix}A_{11} & A_{12} \\ A_{21} & A_{22} \end{bmatrix} \begin{bmatrix}y \\ x \end{bmatrix}+ \begin{bmatrix}B_1 \\ B_2 \end{bmatrix} u \\
y &= \begin{bmatrix}I & 0 \end{bmatrix} \begin{bmatrix}y\\ x \end{bmatrix}, \nonumber
\end{align}
Here $y$ are the states that are measured, and $x$ are the hidden states. Note that the assumption in the second equation is made for notational convenience. For a detailed derivation please see \cite{sean:csm}. 

Now, taking Laplace Transforms of the signals in (\ref{eqn:trans_sys}), we get
\begin{align}\label{eqn:laplace_sys}\begin{bmatrix}sY \\ sX \end{bmatrix} &= \begin{bmatrix}A_{11} & A_{12} \\ A_{21} & A_{22} \end{bmatrix} \begin{bmatrix}Y \\ X \end{bmatrix}+ \begin{bmatrix}B_1 \\ B_2 \end{bmatrix} U.
\end{align}
Solving for X in the second equation of \ref{eqn:laplace_sys} gives $$X=(sI-A_{22})^{-1} A_{21}Y + (sI-A_{22})^{-1}B_2U$$
Substituting into the first equation of (\ref{eqn:laplace_sys}) we get,
$$sY = WY + VU,$$ where $W=A_{11} + A_{12}(sI-A_{22})^{-1}A_{21}$ and $V=A_{12}(sI-A_{22})^{-1}B_2 + B_1$. 
Let $D$ be a diagonal matrix with the diagonal entries of $W$. Then, $$(sI-D)Y = (W-D)Y+VU.$$Now we can rewrite this equation as, \begin{equation}\label{eqn:dsf} Y=QY+PU, \end{equation} where 
$$Q = (sI-D)^{-1}(W-D)$$ and $$P=(sI-D)^{-1}V.$$ 
The matrix $Q$ is a matrix of transfer functions from $Y_i$ to $Y_j$, $i\ne j$, or relating each measured signal to the other measured signals. A nonzero entry in $Q_{ji}$ says that the signal $Y_i$ affects the signal $Y_j$ either directly or through some hidden states. Note that $Q$ is zero on the diagonal and either zero or a strictly proper transfer function on the off diagonal.  The matrix $P$ is a matrix of zeros or strictly proper transfer functions from each input to each output without depending on any additional measured states. Together, the pair $(Q(s),P(s))$ is called the {\em dynamical structure function} for system (\ref{eqn:trans_sys}). The transfer function matrix for this system is given by $$G = (I-Q)^{-1}P = C(sI-A)^{-1}B.$$ Hence, DSF can also be seen as an interconnection of the systems $Q$ and $P$ as shown in Figure \ref{fig:dsf as feedback}. Also, note that if $Q=0$, $G=P$.

\begin{figure}[ht]
\centering
\includegraphics[scale=.7]{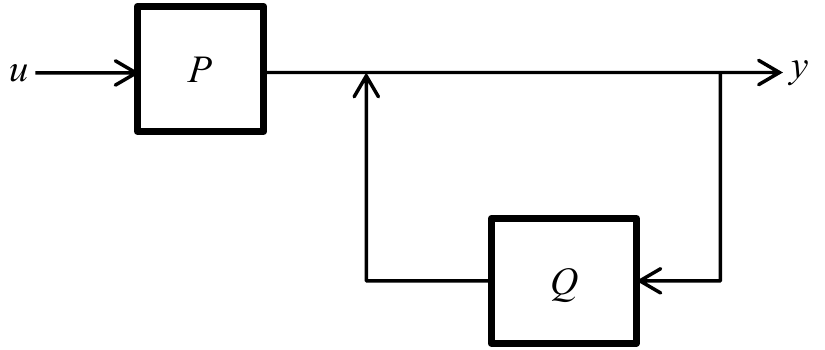}
\caption{DSF can be viewed as an interconnection of two systems characterized by the transfer function matrices $Q$ and $P$, where $Q$ is a hollow transfer function matrix. The transfer function from $u$ to $y$ is given by $G = (I-Q)^{-1}P$. }
\label{fig:dsf as feedback}
\end{figure}

\begin{example}
Let us consider a system given by the following state space equations
\begin{align*}
\dot{\begin{bmatrix}x_1 \\ x_2 \\ x_3\end{bmatrix}} &= \begin{bmatrix}1 & 0 & 3 \\  0 & 2 & 3 \\1 &3 &2\end{bmatrix}\begin{bmatrix}x_1 \\ x_2 \\ x_3\end{bmatrix} + \begin{bmatrix} 1 & 0\\0 & 1\\0 & 0\end{bmatrix}\\
y &= \begin{bmatrix}1 & 0 & 0\\0 & 1 & 0\end{bmatrix}\begin{bmatrix}x_1 \\ x_2 \\ x_3\end{bmatrix}
\end{align*}
The corresponding DSF is given by 
\begin{align*}Q &= \left(\begin{array}{cc} 0 & -\frac{9}{ - s^2 + 3\, s + 1}\\ \frac{3}{\left(s + 1\right)\, \left(s - 5\right)} & 0 \end{array}\right) \text{ and } \\P &= \left(\begin{array}{cc} -\frac{s - 2}{ - s^2 + 3\, s + 1} & 0\\ 0 & \frac{1}{2\, \left(s + 1\right)} + \frac{1}{2\, \left(s - 5\right)} \end{array}\right).\end{align*}
Here, $x_1$ and $x_3$ are the manifest states, and $x_3$ is the hidden shared state. $x_3$ is called the shared state because it is shared between the two links in $Q$.
\end{example}

In this paper, the structure of a controller is defined as a sparsity constraint on the $Q$ matrix; we assume, for the ease of exposition, that the $P$ matrix to be diagonal. We will use the binary matrices $(Q^{bin}, P^{bin})$ to represent the sparsity of the desired controller. The $(i,j)^{th}$ element of $Q^{bin}$, $q^{bin}_{ij} = 1$ if the $j^{th}$ controller unit can communicate with the $i^{th}$ controller unit. Similarly, $p^{bin}_{ij}=1$ if the $j^{th}$ plant unit communicates with the $i^{th}$ controller unit.  $K^{bin}$ represents a structural constraint on the transfer function of the controller. 

\begin{example}
Using this notation, the desired controller in Figure 1 is given by:
$$P^{bin}=\begin{bmatrix}1& 0 & 0  \\ 0 & 1& 0  \\ 0 & 0 & 1 \end{bmatrix}$$ and 
$$Q^{bin}=
\begin{bmatrix}0 & 0 & 1 \\ 1 & 0 &  0 \\ 0 & 1 & 0\end{bmatrix}.
$$
\renewcommand\arraystretch{1.2}
Let us assume that the transfer function $Q_{ij} = q_{ij}$ if $Q^{bin}_{ij}$ = 1, and $Q_{ij} =0$ otherwise, and similarly $P_{ij} = p_{ij}$ if  $P^{bin}_{ij}$ = 1, and $P_{ij} =0$ otherwise. The corresponding transfer function matrix for this controller is given by, $(Q^{bin}, P^{bin})$
\begin{align*}
K &= (I-Q_k)^{-1}P_k \\
&=\left[\begin{array}{ccc} 
-\frac{{p_{11}}}{{q_{13}}\, {q_{21}}\, {q_{32}} - 1} 				& -\frac{{p_{12}}\, {q_{13}}\, {q_{32}}}{{q_{13}}\, {q_{21}}\, {q_{32}} - 1} 		& -\frac{{p_{13}}\, {q_{13}}}{{q_{13}}\, {q_{21}}\, {q_{32}} - 1}\\ 
-\frac{{p_{11}}\, {q_{21}}}{{q_{13}}\, {q_{21}}\, {q_{32}} - 1} 			& -\frac{{p_{12}}}{{q_{13}}\, {q_{21}}\, {q_{32}} - 1} 						& -\frac{{p_{13}}\, {q_{13}}\, {q_{21}}}{{q_{13}}\, {q_{21}}\, {q_{32}} - 1}\\ 
-\frac{{p_{11}}\, {q_{21}}\, {q_{32}}}{{q_{13}}\, {q_{21}}\, {q_{32}} - 1} 	& -\frac{{p_{12}}\, {q_{32}}}{{q_{13}}\, {q_{21}}\, {q_{32}} - 1} 				& -\frac{{p_{13}}}{{q_{13}}\, {q_{21}}\, {q_{32}} - 1} 
\end{array}\right]
\end{align*}

We can see that this transfer function matrix is full, hence this controller cannot be obtained by placing binary constraints on the transfer function matrix. 

Quadratic Invariance results presented in \cite{lall:qi} provide a method to place other types of constraint on the transfer function. For the structure given in this example the constraints are as follows: 
\begin{align} \label{constraint}
\frac{k_{21}}{k_{11}} = \frac{k_{32}}{k_{13}}, \frac{k_{31}}{k_{21}} = \frac{k_{32}}{k_{22}}, \text{ and } \frac{k_{12}}{k_{32}} = \frac{k_{13}}{k_{33}}
\end{align}
Let us assume that plant has the structure as shown in Figure \ref{fig:ctrl_required}. If $\bar{p}_{ij}$ and $\overline{q}_{ij}$ represents the transfer functions on the DSF of the plant, the transfer function matrix for the plant is given by 
$$G = \left[\begin{array}{ccc} -\frac{\bar{p}_{11}\bar{q}_{12}\bar{q}_{32}}{\bar{q}_{12}\bar{q}_{31}\bar{q}_{32} - 1}      &     -\frac{\bar{p}_{22}}{\bar{q}_{12}\bar{q}_{31}\bar{q}_{32} - 1}  &  - \frac{\bar{p}_{33}\bar{q}_{12}}{\bar{q}_{12}\bar{q}_{31}\bar{q}_{32} - 1}\\
     -\frac{\bar{p}_{11}\bar{q}_{32}}{\bar{q}_{12}\bar{q}_{31}\bar{q}_{32} - 1} & -\frac{\bar{p}_{22}\bar{q}_{31}\bar{q}_{32}}{\bar{q}_{12}\bar{q}_{31}\bar{q}_{32} - 1} &         -\frac{\bar{p}_{33}}{\bar{q}_{12}\bar{q}_{31}\bar{q}_{32} - 1}\\
           -\frac{\bar{p}_{11}}{\bar{q}_{12}\bar{q}_{31}\bar{q}_{32} - 1} &     -\frac{\bar{p}_{22}\bar{q}_{31}}{\bar{q}_{12}\bar{q}_{31}\bar{q}_{32} - 1} & -\frac{\bar{p}_{33}\bar{q}_{12}\bar{q}_{31}}{\bar{q}_{12}\bar{q}_{31}\bar{q}_{32} - 1}\end{array}\right].$$
By computing the product $Z = KGK$ we can see that $$\frac{z_{21}}{z_{11}} \ne \frac{z_{32}}{z_{13}}.$$ This violates the constraints given in Equation (\ref{constraint}), hence, the plant and the controller are not quadratically invariant and the algorithm in \cite{lall:qi}  cannot be used to construct such controllers.
\end{example}
\renewcommand\arraystretch{1}

\section{Main Result}\label{sec:main}

In this section, we present a procedure to design a controller $(Q,P)$ with a structure given by $(Q^{bin}, P^{bin})$ to stabilize a plant with the transfer function matrix $G$. The procedure is as follows:
\\{\bf Procedure $\mathbb{P}$}
\begin{enumerate}
\item Choose an undesigned link $p_{ij}$ such that $p^{bin}_{ij} = 1$
\item Design $p_{ij}$ to stabilize $g_{ji}$ such that there is no pole zero cancellation in $PG$. That is, the controller link is designed such that it stabilizes the transfer function it sees, and there is no pole-zero cancellation.
\item After adding $p_{ij}$, if the closed loop system $(G,P)$ is still unstable, repeat for all $p_{xy}$,  $p^{bin}_{xy}=1$.
\item If the closed loop system $S$, formed by adding $P$ in feedback with $G$, is still unstable, add links in $Q^{bin}$ such that there is no pole-zero cancellation between $Q$ and $S$.
\end{enumerate}

\begin{theorem} 
Given a transfer function matrix, $G$, and a desired signal structure for a feedback controller characterized by $(Q^{bin}, P^{bin})$, Procedure  $\mathbb{P}$ either delivers a stabilizing controller with the desired structure or no such controller exists.
\end{theorem}

This theorem says that if the controller obtained using this procedure does not stabilize the plant, then there is no controller of the given structure that can stabilize it. Hence, this procedure provides a test for the existence of a structured stabilizing controller, and if such a controller exists, it synthesizes a nominal stabilizing controller that meets the structural constraint. Before proving this theorem, we will prove some lemmata.

\begin{lemma}
\label{lem:cannotAffect}
Let $K$ be the controller transfer function. A link $k_{ij}$ cannot affect a mode of the plant $G$ that is not observable or controllable from this link.
\end{lemma}

\begin{proof}
Let, $$G=\left[\begin{array}{c|c} A & B \\ \hline C & D \\ \end{array}\right] \text{ and } k_{ij}=\left[\begin{array}{c|c}A_k & B_k \\ \hline C_k & 0  \end{array}\right].$$ Since we are only adding one link, both of these systems are SISO.
Using the Kalman decomposition on $G$, we can transform it such that 
$$A = \begin{bmatrix} A_{co} & 0 & A_{\times o} & 0 \\ 
A_{c \times} & A_{c\bar{o}} & A_{\times \times} & A_{\times \bar{o}} \\ 
0 & 0 & A_{\bar{c}o} & 0 \\ 
0 & 0 & A_{\bar{c} \times} & A_{\bar{c} \bar{o}} \end{bmatrix}, 
B = \begin{bmatrix}B_{co} \\ B_{\bar{c}o} \\ 0 \\ 0\end{bmatrix}$$ 
$$C = \begin{bmatrix} C_{co} & 0 & C_{c \bar{o}} & 0\end{bmatrix}, \text{ and } D=d.$$ Here, the eigenvalues of $A_{c\bar{o}}$, $A_{\bar{c}o}$, and $A_{\bar{c}\bar{o}}$ are the modes of $G$ that are unobservable, uncontrollable, and both respectively from feedback link $k_{ij}$.

The closed loop modes are given by the eigenvalues of the following matrix: 
\begin{align*}A_{cl} &= \begin{bmatrix}A & BC_k \\ B_kC & A_k+B_kDC_k\end{bmatrix}\\ 
	&= \left[\begin{array}{ccccc} 
A_{co} & 0 & A_{\times o} & 0 & B_{co}C_k \\ 
A_{c \times} & A_{c\bar{o}} & A_{\times \times} & A_{\times \bar{o}} & B_{c \bar{o}}C_k \\
0 & 0 & A_{\bar{c}o} & 0 & 0\\
0 & 0 & A_{\bar{c} \times} & A_{\bar{c} \bar{o}} & 0 \\
B_kC_{co} & 0 & B_k C_{c\bar{o}} & 0 & A_k+B_kDC_k
 \end{array}\right] \end{align*}
Transforming this matrix using the permutation $$T=\begin{bmatrix}0 & 1 & 0 & 0 & 0\\1 & 0 & 0 & 0 & 0\\0& 0 &0 & 0& 1 \\0 & 0 & 0 & 1 & 0\\0 & 0 & 1 & 0 & 0\end{bmatrix},$$ we get,
\begin{align*}
A_{clT} &= TA_{cl}T'\\ &=
\left[\begin{array}{ccccc} 
A_{c\bar{o}}&A_{c \times} 	& B_{c \bar{o}}C_k  	& A_{\times \bar{o}} 	& A_{\times \times}\\
0 		& A_{co}		& B_{co}C_k 		& 0 				& A_{\times o} \\ 
0		& B_kC_{co} 	& A_k+B_kDC_k	 	& 0 				& B_k C_{c\bar{o}} \\
0 		& 0			&  0				& A_{\bar{c} \bar{o}} 	& A_{\bar{c} \times} \\
0 		& 0 			& 0		 		& 0 				& A_{\bar{c}o}
 \end{array}\right]
\end{align*}
We can see that $A_{clT}$ is block triangular, and the uncontrollable or unobservable modes, namely the eigenvalues of $A_{\bar{c}o}, A_{c\bar{o}}$, and $A_{\bar{c}\bar{o}}$, are not affected by the choices of $A_k, B_k$, or $C_k$.
\end{proof}

This result shows that when a controller link is added to the system such that it stabilizes all the modes that it can control and observe, it cannot destabilize other modes of the system that are already stable. Now, the following lemma gives a necessary and sufficient condition for the existence of the controller with transfer function structure $K^{bin}$.

\begin{lemma}
\label{lem:existence}
There exists a controller with pattern $K^{bin}$ that stabilizes a plant $G$ if and only if every unstable mode of $G$ is controllable and observable from at least one link $k_{ij}$, $k^{bin}_{ij} = 1$.
\end{lemma}
\begin{proof}
From Lemma \ref{lem:cannotAffect}, we know that a link in the feedback controller cannot affect the uncontrollable or unobservable modes. Hence, any controller that stabilizes a given $G$ must have links such that all the unstable modes are both controllable and observable from at least one of the controller link. Also, if every unstable mode is controllable and observable from some controller links, these links can stabilize the plant.
\end{proof}

lemmata \ref{lem:cannotAffect} and \ref{lem:existence} allow us to add links in $P$, since adding a link in $P$ cannot change the controllability/observability of the plant for the other links in $P$. However, adding these links might cause the links in $Q$ to lose controllability or observability of some of the modes, because links in $Q$ are added on top of the links in $P$. Also, the links in $Q$ themselves can create controllability/observability issues for subsequent links in $Q$.

Loss of observability/controllability can happen for two reasons: structurally or by exact cancellations. If it happens because of structural reasons, the system stays uncontrollable/unobservable for any choice of $P$ or $Q$ as long as it has the same structure. However, if the problem occurs because of exact cancellations, we can avoid these issues by a proper choice of the transfer function. Lemma \ref{lem:pzcancellation} provides a methodology to design $P$ and $Q$ such that these cancellations are prevented. We will use the following result from \cite{mimostability} to prove the lemma.
\begin{theorem} \label{mimo_stability}
Let $G$, $H$ be proper rational transfer function matrices and suppose that $det[I+G(\infty)H(\infty)]\ne0$. Then all the poles of the transfer function matrix $$W=\begin{bmatrix} (I+HG)^{-1} & -H(I+GH)^{-1} \\ G(I+HG)^{-1} & (I+GH)^{-1}\end{bmatrix}$$ are stable if and only if 
\begin{itemize}
\item $GH$ has no unstable pole-zero cancellation, and
\item all the poles of $(I+GH)^{-1}$ are stable.
\end{itemize}
\end{theorem}
\begin{proof}
See \cite{mimostability} Theorem 5.
\end{proof}

\begin{lemma}\label{lem:pzcancellation}
Loss of controllability/observability can be prevented from each link in $Q$ if pole-zero cancellations are avoided in $PG$ and $QS$. Here, $S$ is the closed loop transfer function that $Q$ observes and controls.
\end{lemma}
\begin{figure}[ht]
\centering
\includegraphics[scale=.8]{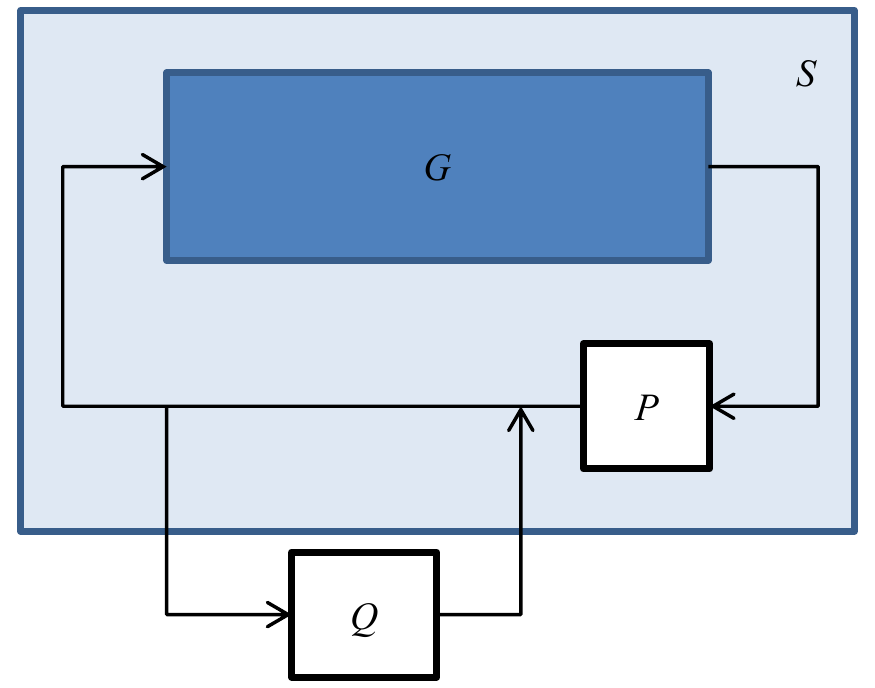}
\caption{After designing $P$, the plant as seen by $Q$ is given by $S=(I-PG)^{-1}$.}
\label{fig:qcl}
\end{figure}

\begin{proof}
The transfer function that $Q$ observes for the closed loop system formed by adding $P$ in feedback with $G$ is given by $S=(I-PG)^{-1}$ as shown in Figure \ref{fig:qcl}. 
Using the Theorem \ref{mimo_stability}, since there is no pole zero cancellations in $PG$, the closed loop system is stable if and only if $S$ is stable. Which says that this transfer function has all the poles of the system. Hence $Q$ observes and controls all the poles of the system after adding all the links in $P$ if there is no pole zero cancellation in $PG$.

Similarly, when adding the links in $Q$ if there is no pole zero cancellation in $QS$ the controllability and observability properties are maintained. That is, if a mode is observable/controllable from a link $Q_{ij}$ for some choices of the other links in the controller, then choosing the links in this fashion will keep the mode observable/controllable from $Q_{ij}.$
\end{proof}

Now we will present the proof of Theorem 1:

\begin{proof}
For every controller link that is added, either in $P$ or $Q$, it stabilizes all the modes that are controllable and observable. Also, by Lemma \ref{lem:cannotAffect}, a newly added link cannot destabilize a mode that was already stable. Hence with every new link added to the system, the number of unstable modes either decreases or stays the same. 

If every unstable mode in the system is controllable and observable by some link, it gets stabilized. If the plant has an unstable mode that is uncontrollable and unobservable from every link in $P$ and $Q$, then by Lemma \ref{lem:existence}, there is no controller with the given pattern that stabilizes the plant. Also, since the added links satisfy the conditions in Lemma \ref{lem:pzcancellation}, if a mode is controllable or observable from a link for any choices previously added links, then it is controllable and observable.
\end{proof}

\section{Specific Examples}
In this section we use Procedure $\mathbb{P}$ to identify plants that are stabilizable or not stabilizable by controllers with some specific structural constraints.

\subsection{Controllers with a cyclic structure}
A cycle in the controller can be represented by the following binary constraints: \begin{align*}
P_{cyl}^{bin}&=\begin{bmatrix}1 & 0 & 0 & 0\\0 & 1 & 0 & \vdots\\0 & 0 & \ddots & 0 \\ 0 & \cdots & 0& 1 \end{bmatrix}_{n\times n} \text{ and },\end{align*}
\begin{align*}Q_{cyl}^{bin}&=\begin{bmatrix}0 & 1 & 0 & 0  & 0\\ 0 & 0 & 1 & 0 & \vdots\\ 0 & 0 & 0 & \ddots & 0\\  0 & 0 &  \cdots & 0 & 1\\ 1 & 0 & 0 & \cdots & 0\end{bmatrix}_{n \times n}.\end{align*} For such constraints on the controller we can prove the following result.
\begin{corollary}
If an $n \times n$ plant is detectable and stabilizable, there always exists a stabilizing controller with the structure $(Q_{cyl}^{bin}, P_{cyl}^{bin})$ .
\end{corollary}
\begin{proof}
When all the links in $P$, and all but the last one in $Q$ is added, all the remaining unstable modes of the system must be observable and controllable from th last link in $Q$. This happens because when adding links in the controller we satisfy the conditions in \ref{lem:pzcancellation} avoiding any pole zero cancellations. Hence, if a link $Q_{i+1,i}$ is added then all the modes that are observable at $y_i$ are also observable at $y_{i+1}$, and all the modes that are controllable from $u_{i+1}$ are also controllable from  $u_i$.
\end{proof}

\subsection{Systems that are not stabilizable by a diagonal controller}
We know that not all plants can be stabilized by a diagonal controllers. To study these systems one might want to generate plants that fall in this category. We can use our results to design such systems.

From Lemma \ref{lem:existence}, we know that a detectable and stabilizable plant can be stabilized by a diagonal controller if and only if a mode of the system that is controllable from input $i$ is also observable at the output $i$. Hence, a plant cannot be stabilized by a diagonal controller if there is a node that is observable only at output $i$ and controllable only from input $j$, $i\ne j$. For example, the following system cannot be stabilized by a diagonal controller:
\begin{align*}
\dot{x} &= \begin{bmatrix}1 & 0 & 0\\ 1 & 2 & 3 \\ 1 & 0 & 3\end{bmatrix}x + \begin{bmatrix}1 & 0\\0 & 1 \\ 0 & 0\end{bmatrix}\begin{bmatrix}u_1 \\ u_2\end{bmatrix}\\
\begin{bmatrix}y_1 \\ y_2\end{bmatrix}&=\begin{bmatrix}1 & 0 & 0\\0 & 1 & 0\end{bmatrix} x
\end{align*}
This system has the modes at \{1,2,3\}. Using the Popov-Belevitch-Hautus (PBH) tests for controllability and observability, we can see that the mode 3 is controllable only from input $u_1$ and observable only at output $y_2$. Hence a diagonal controller cannot satisfy the condition given in Lemma \ref{lem:existence}.
\section{Conclusion}
In this paper, we presented an algorithm to construct stabilizing controllers with a given signal structure. We also showed that if the procedure fails to produce a stabilizing controller, the plant cannot be stabilized with a controller with the given structure. 

We note that this procedure might not be a practical method for generating stabilizing controllers. This method does not provide any optimality guarantees. Also, if synthesis techniques LQG is used to construct the controller links, the order of the transfer function on these links grows exponentially. Hence, we need to develop a controller synthesis technique that produce a low order controller.  These issues will be addressed in the future research. Nevertheless, this paper introduces a new kind of decentralized control problem which is very important for networked systems, and gives a nominal solution for it. 

\section{Acknowledgment}
We gratefully acknowledge the generous support of AFRL grants FA8750-09-2-0219 and FA8750-11-1-0236. 

\bibliography{cdc12}{}
\bibliographystyle{plain}
\end{document}